\documentclass[12pt]{amsart}
\usepackage{a4wide}
\usepackage{xspace}

\usepackage{amsmath,amssymb,amscd,amsthm,latexsym}
\usepackage[all,cmtip]{xy}
\newdir{ >}{{}*!/-5pt/@{>}}
\usepackage{comment}

\theoremstyle{plain}
    \newtheorem{thm}{Theorem}[section]
    \newtheorem{lem}[thm]   {Lemma}
    
    \newtheorem{prop}[thm]  {Proposition}

\theoremstyle{definition}
    \newtheorem{defn}[thm]  {Definition}

    \newtheorem{ex}[thm]{Example}
    \newtheorem{rem}[thm]{Remark}

\newcommand{\thmref}[1]{Theorem~\ref{#1}}

\newcommand{\an}[1]{\langle{#1}\rangle}
\newcommand{\wt}{\widetilde}
\newcommand{\wh}{\widehat}

\newcommand{\xra}{\xrightarrow}
\newcommand{\Hom}{\mathrm{Hom}}
\newcommand{\Ker}{\mathrm{Ker}}

\newcommand{\del}{\partial}
\newcommand{\bz}{\beta^{\Z/2}}

\newcommand{\pr}{\mathrm{pr}}
\newcommand{\ind}{\mathrm{ind}}
\newcommand{\per}{\mathrm{per}}
\newcommand{\ord}{\mathrm{ord}}
\newcommand{\Brt}{\mathrm{Br}}

\def\Q{{\mathbb{Q}}}

\def\C{{\mathbb{C}}}
\def\R{{\mathbb{R}}}

\def\Z{{\mathbb{Z}}}

\def\Hom{\operatorname{Hom}}

\def\im{\operatorname{Im}}
\newcommand{\spc}{$\textrm{spin}^{c}$\xspace}
\renewcommand{\ind}{\mathrm{ind}}

\usepackage{color}

\title[Topological Period-Index Conjecture]
{The Topological Period-Index Conjecture for spin$^c$ $6$-manifolds}

\author{Diarmuid Crowley}
\author{Mark Grant}

\thanks{}

\address{School of Mathematics \& Statistics,
The University of Melbourne,
Parkville, VIC, 3010,
Australia}
\email{dcrowley@unimelb.edu.au}

\address{Institute of Mathematics,
University of Aberdeen,
Fraser Noble Building,
Meston Walk,
Aberdeen AB24 3UE, UK}
\email{mark.grant@abdn.ac.uk}

\begin{document}

\begin{abstract}
The Topological Period-Index Conjecture is an hypothesis which
relates the period and index of elements of the
cohomological
Brauer group of a space.
It was identified by Antieau and Williams as a topological analogue
of the Period-Index Conjecture for function fields.

In this paper we show that the Topological Period-Index Conjecture holds and is in general sharp for spin$^c$ $6$-manifolds.
We also show that it fails in general for $6$-manifolds.
\end{abstract}

\maketitle



\section{Introduction}
This paper is about the {\em Topological Period-Index Problem} (TPIP),
which was identified by Antieau and Williams \cite{AW1, AW2}
as an important analogue of period-index problems in algebraic geometry.
We give a brief introduction to the TPIP and refer
the reader to \cite{AW1, AW2} for more information.

Let $X$ be a connected space with the homotopy type of finite $CW$-complex.
%
The cohomological Brauer group of $X$ is defined to be
the torsion subgroup of its third integral cohomology group:
$$ \mathrm{Br}'(X) := TH^3(X)$$
Here and throughout integer coefficients are
omitted.
For $\alpha \in \mathrm{Br}'(X)$, the period of $\alpha$ is defined to
be the order of $\alpha$,
$$ \mathrm{per}(\alpha) := \ord(\alpha). $$
%

Let $PU(n) := U(n)/U(1)$ be the $n$-dimensional projective unitary group, which is the
quotient of the unitary group $U(n)$ by its centre.
By a Theorem of Serre \cite[Corrollaire 1.7]{G},
every class $\alpha \in TH^3(X)$
arises as the obstruction to lifting
the structure group of some principal $PU(n)$-bundle $P \to X$ to the group $U(n)$.
In this case one writes
$ \alpha = \delta(P)$
and defines the index of $\alpha$ by
$$ \mathrm{ind}(\alpha) : = \mathrm{gcd} \bigl( n : \alpha = \delta(P)
\text{~for a $PU(n)$-bundle $P$} \bigr),$$
so that the index defines the homotopy invariant function
$$ \ind \colon TH^3(X) \to \Z,
\quad \alpha \mapsto \ind(\alpha).$$
%
From the definitions, one sees that $\per(\alpha)\, |\, \ind(\alpha)$
and by \cite[Theorem 3.1]{AW1} the primes dividing
$\per(\alpha)$ and $\ind(\alpha)$ coincide.
The TPIP
is the problem of relating the index of a class $\alpha$
to its period and properties of $X$, like its dimension.

To investigate the TPIP,
Antieau and Williams \cite[Straw Man]{AW2} formulated what is often called
the {\em Topological Period-Index Conjecture} (TPIC) for $X$:\\
{\bf TPIC:} {\em If $X$ is homotopy equivalent to $CW$-complex of dimension $2d$
and if $\alpha \in \Brt'(X)$ then}
$$ \ind(\alpha) \,|\, \per(\alpha)^{d-1}. $$
{\bf Warning:} The TPIC should be regarded as an {\em hypothesis} for investigating
the TPIP and {\em not as a conjecture},
in the usual sense of the word.\\
Indeed, while the obstruction theory developed by
Antieau and Williams \cite[Theorem A]{AW2}
shows that the TPIC holds for any $4$-dimensional complex,
they also prove that the TPIC fails in general for $6$-dimensional complexes,
but at most by a factor of two.

\begin{thm}[{C.f.\ \cite[Theorems A \& B]{AW2}}] \label{thm:tpip6}
Let $X$ be $6$-dimensional, $\alpha \in \Brt'(X)$ have period $n$
and set $\epsilon(n) := \gcd(n, 2)$.
Then $\ind(\alpha) \, | \, \epsilon(n)n^2$.

Moreover, if $X$ is a
$6$-skeleton of the Eilenberg-Mac Lane space $K(\Z/2, 2)$ and
we take the generator $\alpha \in H^3(X) = \Z/2$ (so that $\per(\alpha) = 2$),
then $\ind(\alpha) = 8 > \per(\alpha)^2$.
\end{thm}

An important motivation for Antieau and Williams in identifying the TPIC
was the Algebraic Period-Index Conjecture (APIC) which
was identified in the work of Colliot-Th\'el\`ene \cite{CT}.
This is a statement in algebraic geometry concerning the Brauer group
of certain algebras $A$.
When $A = \C(V)$ is the function field of a smooth complex variety $V$ then
the APIC for $\C(V)$ implies the TPIC for $V$.
When the variety $V$ has complex dimension $d=1$ the APIC is trivially
true, it was proven for $d = 2$ by de Jong \cite{deJong}
and for $d \geq 3$
we have the {\em Antieau-Williams alternative}:
\begin{enumerate}
\item[(A)] Either there exits a $V$
violating the TPIC, in which case the APIC fails in general,
\item[(B)] Or every $V$ satisfies the TPIC (in which case we have identified
an {\em a priori} new topological property of smooth complex varieties).
\end{enumerate}
{\em In this paper we show that for $d=3$ the latter statement holds.
This may be regarded as evidence for the APIC in complex dimension $3$.}
%

A smooth complex projective variety $V$
is in particular a manifold: here
and besides Remark \ref{rem:bdry},
we use the
word ``manifold'' to mean ``closed smooth manifold''.
Recall that a manifold $M$ admits a \spc structure
if it is orientable and
the second Stiefel-Whitney class of $M$
has an integral lift.
For example, every variety $V$ as above admits a \spc structure.
More generally, it is well known that
a $6$-manifold admits a \spc structure if and only if it admits
an almost complex structure
(as can be easily deduced from results in \cite{Massey}).

\begin{thm} \label{thm:tpic6}
The Topological Period-Index Conjecture holds for
\spc $6$-manifolds.
%
\end{thm}
As we explain in Section \ref{s:TPIC},
Theorem \ref{thm:tpic6} is an elementary consequence of
results of Antieau and Williams \cite{AW2} and the following

\begin{thm}\label{thm:main}
Let $N$ be a closed \spc $6$-manifold and let $x\in H^2(N;\Z/2)$.
Then there exists a class $e_x \in H^2(N)$ such that
\[
\beta^{\Z/2}(x^2) = \beta^{\Z/2}(x)e_x \in H^5(N),
\]
where $\beta^{\Z/2}:H^*(N;\Z/2)\to H^{*+1}(N)$ denotes the mod 2 Bockstein.
\end{thm}

To discuss the 
TPIP
further for $6$-manifolds
we recall that Teichner \cite{Teichner} has already constructed
$6$-manifolds $N$ with $x \in H^2(N; \Z/2)$
such that $\bz(x^2) \neq 0$.
The manifolds in Teichner's
examples are all the total-spaces of $2$-sphere bundles over
$4$-manifolds,
where the class $x$ restricts to a generator of $H^2(S^2; \Z/2)$.
We call pairs $(N, x)$ coming from Teichner's examples {\em Teichner pairs}
(see Definition \ref{defn:TP}) and investigating their construction
we prove

\begin{thm} \label{thm:T1}
For a Teichner pair $(N, x)$, let $\alpha := \bz(x) \in TH^3(N)$.
\begin{enumerate}
\item If the base $4$-manifold of a Teichner pair $(N, x)$ is orientable,
then $N$ is \spc, 
$\per(\alpha) = 2$ and $\ind(\alpha) = 4$.
\item
There exist Teichner pairs $(N, x)$ over non-orientable $4$-manifolds
with $\per(\alpha) = 2$ but $\ind(\alpha) = 8$.
\end{enumerate}
\end{thm}

%
\noindent
Summarising Theorems \ref{thm:tpic6} and \ref{thm:T1} we obtain
the following result on the
TPIP
for $6$-manifolds.

\begin{thm}
The TPIC fails in general for $6$-manifolds but it holds and is in general sharp
for \spc $6$-manifolds.
\end{thm}

\begin{rem} \label{rem:spin^c}
One may view Theorem \ref{thm:main} as giving a cohomological obstruction to a closed $6$-manifold admitting a \spc structure. For instance, we do not currently know how to prove that the Teichner manifold $N$ appearing in Theorem \ref{thm:T1} (2) (and Proposition \ref{prop:T}) is not \spc,
except by invoking Theorem \ref{thm:main}.
\end{rem}

\begin{rem}
The non-vanishing of $\bz(x^2)\in H^5(N)$ is related to various non-realizability phenomena, for which the examples in \cite{Teichner} are of minimal dimension.
For example, $\bz(x^2)$ vanishes if $x\in H^2(N;\Z/2)$ can be realized as the second Stiefel-Whitney class $w_2(E)$ of some real vector bundle $E$ over $N$, since $w_2(E)^2$ is the mod $2$ reduction of the integral class $p_1(E)$, the first Pontrjagin class.

It is a classical result of Thom that $\bz(x^2)$ vanishes if the Poincar\'e dual of $x$ in $H_4(N; \Z/2)$ is realized as the fundamental class of an embedded $4$-manifold in $N$ \cite{Thom}.
More recently, in \cite{G-Sz}
the second author and Sz\H ucs showed that $\bz(x^2)$ vanishes if the
Poincar\'e dual of $x$ is realized by the fundamental class of an immersion of
a $4$-manifold in $N$ and more precisely that the Poincar\'e dual of $\bz(x^2)$ is realized by the singular set of a generic smooth map realizing the Poincar\'e dual of $x$.
Notwithstanding Remarks \ref{rem:spin^c} and \ref{rem:K-alpha},
the geometric significance of the condition $\bz(x^2)\notin \bz(x) H^2(N)$
appearing in Section \ref{s:TPIC}
remains somewhat mysterious.
\end{rem}

\begin{rem} \label{rem:K-alpha}
The TPIP also arises in twisted K-theory, where classes
$\alpha \in TH^3(X)$ define the twisting used to define the
$K$-groups, $K^*_\alpha(X)$,  of $\alpha$-twisted vector bundles over $X$ \cite{D-K}.
For $\alpha \in TH^3(X)$ and
$i \colon \ast \to X$ is the inclusion of a point, by \cite[Proposition 2.21]{AW1}. we have
$$ i^*(K_\alpha^0(X)) = \ind(\alpha)K^0(\ast) = \ind(\alpha) \Z.$$
Hence the index of $\alpha$ is the index of the intersection
$\bigcap_{i=1}^\infty \Ker(d_i) \subseteq H^0(X; K^0) \cong \Z$,
where $d_i \colon H^0(X; K^0) \to H^i(X; K^{i-1})$ is the
$i$th differential in the
twisted Atiyah-Hirzebruch
spectra sequence computing $K^*_\alpha(X)$.

This perspective is behind the index formula
\cite[Theorem A]{AW2},
which we use in Section 2, and also the recent work of Gu \cite{Gu} on
the TPIP for 8-complexes.  Gu shows that the 3-primary
TPIP for 8-complexes involves controlling
$\beta^{\Z/3}(x^3)/\beta^{\Z/3}(x)H^4(X)$ for classes $x \in H^2(X; \Z/3)$,
just as the TPIP for $6$-complexes involves controlling
$\beta^{\Z/2}(x^2)/\beta^{\Z/2}(x)H^2(X)$ for classes $x \in H^2(X; \Z/2)$.
We expect that the methods of this paper involving the integrality of Wu
classes and the bi-linear algebra of Section \ref{ss:Martin}, will generalise
to combine with the work of Gu and prove the TPIC for odd order Brauer classes over
orientable $8$-manifolds.
\end{rem}

\begin{rem} \label{rem:bdry}
It is natural to wonder whether the singular spaces $Z$ underlying singular complex 3-dimensional projective varieties satisfy the TPIC.
In this direction, we note that the complement of the singular set in $Z$
can often be compactified to give a \spc manifold with boundary $(N, \del N)$.
The arguments of this paper can be generalised to prove
that if $(N, \del N)$ is a compact \spc manifold with
boundary where the first Chern class of $N$ vanishes on $\del N$ and
$TH_1(\del N) \otimes \Z/2 = 0$, then the TPIC holds for quotients $N/\del N$.
As a consequence we believe that the TPIC holds for singular spaces
underlying certain complex $3$-dimensional varieties with isolated conical singularities.
\end{rem}

{\bf Organisation:}
The rest of this paper is organised as follows.
In Section \ref{s:TPIC}
we prove Theorem \ref{thm:tpic6}
assuming Theorem \ref{thm:main}.
In Section \ref{s:BB} we establish
some preliminary results about linking pairings and bilinear forms.
In Section \ref{s:proof} we prove Theorem \ref{thm:main}
and in Section \ref{s:TE} we discuss Teichner's examples
and prove Theorem \ref{thm:T1}.

{\bf Acknowledgements:}
We would like to thank Ben Antieau and Ben Williams for helpful
comments and for pointing out mistakes in early stages of our work; Ben Martin for providing
a proof of Lemma \ref{lem:Martin};
Christian Haesemeyer for guidance with the Algebraic
Period-Index Conjecture;
and Jim Davis and Mark Powell for advice regarding the linking pairing.
Finally we thank the anonymous referee of an earlier version
of this paper for comments which have improved the exposition.

\section{The topological period index conjecture
for \spc $6$-manifolds } \label{s:TPIC}
In this section we prove
that the Topological Period-Index Conjecture holds for \spc 6-manifolds.
This is an elementary consequence of Theorem \ref{thm:main}
and results in \cite{AW2}.

Let $\alpha \in \Brt'(X) = TH^3(X)$ with $\ord(\alpha) = n$
and let

$$\beta^{\Z/n} \colon H^*(X; \Z/n) \to H^{*+1}(X) $$
be the mod~$n$ Bockstein, which lies in the exact sequence
$$ H^*(X; \Z/n) \xra{~\beta^{\Z/n}~}
H^{*+1}(X) \xra{~\times n~} H^{*+1}(X).
$$
As $\ord(\alpha) = n$,
we see that $\alpha = \beta^{\Z/n}(\xi)$ for some $\xi \in H^2(X; \Z/n)$.
We consider the Pontrjagin Square
$$ P_2 \colon H^2(X; \Z/2m) \to H^4(X; \Z/4m) $$
and following Antieau and Williams define $\wt Q(\xi) \in H^5(X)/ \alpha H^2(X)$
by the equation
$$ \wt Q(\xi) :=
\begin{cases}
\bigl[ \beta^{\Z/n}(\xi^2) \bigr]  & \text{$n$ is odd,} \\
\bigl[\beta^{\Z/2n}(P_2(\xi)) \bigr] & \text{$n$ is even,}
\end{cases}$$
where $[\gamma] \in H^5(X)/\alpha H^2(X)$ denotes
the coset of $\gamma  \in H^5(X)$.
By \cite[Theorem A]{AW2},
the element $\wt Q(\xi)$ depends only on $\alpha$
and when $X$ is a $6$-dimensional CW-complex,
\begin{equation*} \label{eq:AWe}
\ind(\alpha) = \ord(\wt Q(\xi)) \per(\alpha).
\end{equation*}
Hence to verify the topological period-index conjecture in dimension 6,
it suffices to show that $\ord(\wt Q(\xi)) | n$;
i.e.~$n \wt Q(\xi) = 0$.
For this we consider the following commutative diagram,
\[ \xymatrix{
H^2(X; \Z/2k) \ar[d]^{\rho_2} \ar[r]^(0.6){\beta^{\Z/2k}} &
H^3(X) \ar[d]^{\times k} \\
H^2(X; \Z/2) \ar[r]^(0.575){\beta^{\Z/2}} &
H^3(X),
 } \]
where $\rho_2$ denotes reduction modulo $2$ and the diagram commutes as a consequence of the following commutative diagram of coefficient short exact sequences:
\[ \xymatrix{
\Z \ar[d]_{\times k} \ar[r]^{\times 2k} &
\Z \ar[d]^{=} \ar[r]^(0.425){\rho_{2k}} &
\Z/2k \ar[d]^{\rho_2} \\
\Z \ar[r]^{\times 2} &
\Z \ar[r]^(0.45){\rho_2} &
\Z/2 }\]
Hence for all $\xi \in H^2(X; \Z/2k)$ we have the equation
\begin{equation} \label{eq:beta}
\beta^{\Z/2}(\rho_2(\xi)) = k\beta^{\Z/2k}(\xi).
\end{equation}

\begin{proof}[Proof of the topological period-index conjecture for
\spc $6$-manifolds]
Let $(N, c_1)$ be a \spc $6$-manifold,
$\alpha \in \Brt'(N)$ have order $n$
and choose $\xi \in H^2(N; \Z/n)$ such that $\alpha = \beta^{\Z/n}(\xi)$.
If $n$ is odd then $n \wt Q(\xi) =0 $ and so by
\cite[Theorem A (3)]{AW2} the topological period-index conjecture holds for $\alpha$.
If $n = 2m$ then set
$$ x := \rho_2(\xi) \in H^2(N; \Z/2).$$
By Theorem \ref{thm:main}, there is a $y \in H^2(N)$ such that
$\beta^{\Z/2}(x^2) = \beta^{\Z/2}(x) y$.
Applying \eqref{eq:beta} we obtain
\begin{equation} \label{eq:beta2}
\bigl[ \beta^{\Z/2}(x^2) \bigr] =
\bigl[ \beta^{\Z/2}(x) y \bigr] =
\bigl[ m \beta^{\Z/2m}(\xi) y \bigr] =
m \bigl[ \alpha y \bigr] = 0
\in H^5(N) / \alpha H^2(N).
\end{equation}
Applying \eqref{eq:beta} and \eqref{eq:beta2} we obtain
$$
2m \wt Q(\xi) = 
2m \bigl[ \beta^{\Z/4m}(P_2(\xi)) \bigr] 
=  \bigl[ 2m \beta^{\Z/4m}(P_2(\xi)) \bigr] 
=  \bigl[ \beta^{\Z/2} \bigl( \rho_2(P_2(\xi)) \bigr) \bigr] 
=  \bigl[ \beta^{\Z/2}(x^2) \bigr] 
=  0.
$$
%
%
\end{proof}
%

\newpage

\section{Linking pairings and bilinear forms} \label{s:BB}
In this section we establish some elementary results
used in the proof of Theorem \ref{thm:main}.

\subsection{Some properties of Bockstein homomorphisms}
For a space $X$ and a positive integer $n$ recall that
$$ \beta^{\Z/n} \colon H^*(X; \Z/n) \to H^{*+1}(X)$$
is the Bockstein associated to the coefficient sequence $\Z \longrightarrow \Z \xra{~\rho_n~} \Z/n$.

\begin{lem} \label{lem:Bprod}
Let $x \in H^*(X; \Z/n)$ and $y \in H^i(X)$, and consider $xy \in H^{*+i}(X; \Z/n)$.
Then
$$ \beta^{\Z/n}(x y) = \beta^{\Z/n}(x)y.$$
\end{lem}

\begin{proof}
Let $y\in C^i(X)$ be a cocycle representative for $y$, $\rho_n$ denote reduction modulo $n$ and consider the commutative diagram below, in which the rows are short exact sequences of chain complexes:
\[
\xymatrix{
0 \ar[r] &
C^*(X) \ar[r]^{\times n} \ar[d]_{\cup y} &
C^*(X) \ar[r]^(0.425){\rho_n} \ar[d]_{\cup y} &
C^*(X;\Z/n) \ar[r] \ar[d]_{\cup \rho_n(y)}& 0\\
0 \ar[r] &
C^{*+i}(X) \ar[r]^{\times n} &
C^{*+i}(X) \ar[r]^(0.425){\rho_n} &
C^{*+i}(X;\Z/n) \ar[r] &
0}
\]
Observe that the vertical arrows are chain maps, since the coboundary is a derivation and $y$ is a cocycle. The result now follows from the naturality of connecting homomorphisms. (Compare Brown, \cite[V, 3.3]{Brown}.)
\end{proof}

We also consider the Bockstein homomorphism
$$ \beta^{\Q/\Z} \colon H^*(X; \Q/\Z) \to H^{*+1}(X), $$
which is associated to the coefficient sequence
$\Z \longrightarrow \Q \xra{~\pi~} \Q/\Z$.
Let
$$\iota_n \colon \Z/n \to \Q/\Z$$
be the inclusion defined by
sending $[1] \in \Z/n$ to $[\frac{1}{n}]$
and also write
$$ \iota_n \colon H^*(X; \Z/n) \to H^*(X; \Q/\Z) $$
for the map on homology induced by $\iota_n$.
The commutative diagram of coefficient sequences
\[\xymatrix{
\Z  \ar[d]^{=} \ar[r]^{\times n} &
\Z  \ar[d]^{\times \frac{1}{n}} \ar[r] &
\Z/n \ar[d]^{\iota_n} \\
\Z \ar[r] &
\Q \ar[r] &
\Q/\Z
}\]
gives rise to the equality
\begin{equation} \label{eq:B_and_iota}
\beta^{\Z/n} = \beta^{\Q/\Z} \circ \iota_n
\colon H^{*}(X; \Z/n) \to H^{*+1}(X).
\end{equation}

\subsection{The linking pairings of an oriented manifold} \label{ss:pairing}
Let $G$ and $H$ be a finite abelian groups.
Recall that a bilinear pairing
$$ \phi \colon G \times H \to \Q/\Z$$
is called {\em perfect}
if $g = 0 \in G$ if and only if $\phi(g, h) = 0$ for all $h \in H$
and $h = 0 \in H$ if and only if $\phi(g, h) = 0$ for all $g \in G$.

\begin{rem} \label{rem:perfect}
A useful property of perfect pairings, which we leave the reader to verify,
is that $h_1 = h_2 \in H$ if and only if $\phi(g, h_1) = \phi(g, h_2)$
for all $g \in G$.  An analogous statement holds for $g_1, g_2 \in G$.
\end{rem}

Now let $M$ be a closed, connected, oriented $m$-manifold with
$[M] \in H_m(M)$ the fundamental class of $M$.
%
%
For each $k = 2, \dots, m-2$, the {\em linking pairing} of $M$ is the pairing
$$ b_M \colon TH^{k+1}(M) \times TH^{m-k}(M) \to \Q/\Z,
\quad (f, y) \mapsto  b_M(f,y) := \an{\wt f  y, [M]},$$
where $\wt f \in H^{k}(M; \Q/\Z)$ is any class such that $\beta^{\Q/\Z}(\wt f) = f$.

\begin{lem} \label{lem:bN}
The linking pairing $b_M \colon TH^{k+1}(M) \times TH^{m-k}(M) \to \Q/\Z$ is a perfect pairing
such that for all $w \in H^{k}(M;\Z/n)$ and all $y \in H^{m-k}(M)$
$$ b_M(\beta^{\Z/n}(w),y) = \iota_n( \an{w y, [M]}).$$
\end{lem}

\begin{proof}
That $b_M$ is perfect is well known. The case $m = 2k{+}1$ is part of \cite[Exercise 55]{Da-K}.
The general case follows from results in \cite{S-T}.
Since we did not find a definitive reference in the literature, we give a proof below.

For a finite abelian group $G$, let $G^\wedge \! := \Hom(G, \Q/\Z)$ denote the {\em torsion dual} of $G$. A pairing $\phi:G\times H\to \Q/\Z$ of finite abelian groups induces adjoint homomorphisms
$\wh \phi: H\to G^\wedge, h \mapsto [g \mapsto \phi(g, h)]$ and
$\wh\phi': G\to H^\wedge, g \mapsto [h \mapsto \phi(g, h)]$, and it is easily checked that $\phi$ is perfect if and only if either one of $\wh\phi$ or $\wh\phi'$ is an isomorphism.

Standard properties of cup and cap products give $ \an{\wt f  y, [M]} = \an{\wt f, y \cap [M]}.$
Hence the adjoint homomorphism of $b_M$,
$$ \wh b_M \colon TH^{m-k}(M) \to TH^{k+1}(M)^\wedge,
\quad y \mapsto [f \mapsto b_M(f, y)=\an{\wt f, y \cap [M]}],$$
is equal to the composition 
$\wh \phi_M \circ PD$, where
$PD \colon TH^{m-k}(M) \to TH_{k}(M)$ is the Poincar\'e duality isomorphism
and $\wh \phi_M \colon TH_{k}(M) \to TH^{k+1}(M)^\wedge$ is an adjoint of the pairing
$$ \phi_M \colon TH^{k+1}(M) \times TH_{k}(M) \to \Q/\Z,
\quad (f,b) \mapsto \an{\wt f, b},$$
for $\wt f \in H^{k}(M; \Q/\Z)$ a lift of $f$.
Hence it suffices to prove that
%
%
$\wh \phi_M$ is an isomorphism; equivalently that the other adjoint
$\wh \phi_M' \colon TH^{k+1}(M) \to TH_{k}(M)^\wedge$ is an isomorphism.
Since the finite groups $TH^{k+1}(M)$ and $TH_{k}(M)^\wedge$
have the same order by the Universal Coefficient Theorem,
it suffices to show that $\wh \phi'_M$ is injective.

Suppose that $\wh \phi_M'(f) = 0$ and let $\wt f \in H^{k}(M; \Q/\Z)$
be a lift of $f$.  Then for all $b \in TH_{k}(M)$
$$ \an{\wt f, b} = 0 \in \Q/\Z.$$
Since $\Q/\Z$ is an injective $\Z$-module,
another application of the Universal Coefficient Theorem
gives
$$ H^{k}(M; \Q/\Z) \cong \Hom(H_{k}(M), \Q/\Z) \cong TH_{k}(M)^\wedge \oplus
\Hom(FH_{k}(M); \Q/\Z),$$
where $FH_{k}(M) : = H_{k}(M)/TH_{k}(M)$.  With respect to the above
decomposition we have $\wt f = (0, \bar z)$ for some $\bar z \in \Hom(FH_{k}(M); \Q/\Z)$.
Now $\bar z$ can be lifted to $z \in H^{k}(M; \Q)$ so that
$\wt f - \pi(z) = 0$ but then $f = \beta^{\Q/\Z}(\wt f) = \beta^{\Q/\Z}(\wt f - \pi(z)) = 0$
and so $\wh \phi_M'$ is injective.

The second statement follows directly from the definition of $b_M$ and \eqref{eq:B_and_iota}.
\end{proof}



\subsection{Bilinear forms over $\Z/2$} \label{ss:Martin}
In this subsection we establish a basic fact about symmetric bilinear forms
over $\Z/2$.
Let $V$ be a finitely generated $(\Z/2)$-vector space and let
$$ \lambda \colon V \times V \to \Z/2$$
be a symmetric bilinear form on $V$.  If $V^* := \mathrm{Hom}(V, \Z/2)$
is the dual vector space to $V$, then
the adjoint homomorphism of $\lambda$
is the homomorphism
$$ \wh \lambda \colon V \to V^*,
\quad v \mapsto \bigl(w \mapsto \lambda(v, w)  \bigr).$$
The form $(\lambda, V)$ is {\em nonsingular} if
$\wh \lambda \colon V \to V^*$ an isomorphism.  Notice that the map
$$ \gamma(\lambda) \colon V \to \Z/2,
\quad
v \mapsto \lambda(v, v)$$
is linear since
\begin{multline*}
\lambda(v+w, v+w) = \lambda(v, v) + \lambda(v, w) + \lambda(w, v) + \lambda(w, w) \\
= \lambda(v, v) + 2\lambda(v, w) + \lambda(w, w) = \lambda(v, v) + \lambda(w, w).
\end{multline*}
Thus $\gamma(\lambda) \in V^*$.

\begin{lem} \label{lem:Martin}
For all $\lambda$, $\gamma(\lambda) \in \im(\wh \lambda)$.
\end{lem}

\begin{proof}
For the orthogonal sum of bilinear forms,
$\lambda_0 \oplus \lambda_1$ we have
$$ \gamma(\lambda_0 \oplus \lambda_1) = \gamma(\lambda_0) \oplus \gamma(\lambda_1).$$
The lemma follows since every symmetric bilinear form over a finite field
is isomorphic to the orthogonal sum of the zero form and a nonsingular form.
%
\end{proof}

\begin{rem} \label{rem:diagonal}
Although we will not use this fact, it is worthwhile noting that Lemma \ref{lem:Martin}
is equivalent to the following statement: Let $A$ be a symmetric matrix over $\Z/2$,
then the diagonal of $A$ lies in the column space of $A$.
\end{rem}

\begin{ex} \label{ex:lambda-x}
Let $N$ be a closed, connected, oriented $6$-manifold and $x \in H^2(N; \Z/2)$.
We identify $H^6(N; \Z/2) = \Z/2$ and
for the $(\Z/2)$-vector space
$$ V := TH^2(N)/2 TH^2(N)$$
we define the symmetric bilinear form
$$ \lambda_{x} \colon V \times V \to \Z/2,
\quad ([y], [z]) \mapsto yxz.
$$
By Lemma \ref{lem:Martin}, there is a vector $[d] \in V$ such that
$\wh \lambda_x([d]) = \gamma(\lambda_{x}) \in V^*$.
Hence for any $d_x \in [d] \subset TH^2(N)$ and all $y \in TH^2(N)$,
we have
$$ y^2x = yxy = \lambda_x([y], [y])  = \lambda_x([y], [d_x]) = yxd_x.$$
\end{ex}

%

\section{The proof of \thmref{thm:main}}\label{s:proof}
Let $N$ be a closed, connected, oriented \spc $6$-manifold.
To prove
Theorem \ref{thm:main} it suffices to prove the following:
for $x \in H^2(N; \Z/2)$ and all $y \in TH^2(N)$,
there is a class $e_x \in H^2(N)$ such that
\begin{equation} \label{eq:key}
x^2y = xe_xy \in H^6(N; \Z/2).
\end{equation}
To see this we use the linking pairing of $N$, which is a perfect pairing by Lemma \ref{lem:bN}:
$$ b_N \colon TH^5(N) \times TH^2(N) \to \Q/\Z$$
From \eqref{eq:key} and Lemmas \ref{lem:bN} and \ref{lem:Bprod}, for all $y \in TH^2(N)$ we have
$$ b_N(\beta^{\Z/2}(x^2),y) = \iota_2(\an{x^2y, [N]}) = \iota_2(\an{xe_xy, [N]}) = b_N(\beta^{\Z/2}(xe_x),y)
= b_N(\beta^{\Z/2}(x)e_x,y).$$
Thus $\beta^{\Z/2}(x^2) = \beta^{\Z/2}(x)e_x$, since $b_N$ is perfect;
see Remark \ref{rem:perfect}.


To find $e_x$, we start with $v_2(N)$, the second Wu class of $N$.
Since $N$ is orientable, $v_2(N)$
coincides with $w_2(N)$, the second Stiefel-Whitney class of $N$.
Since $N$ is \spc the class $w_2(N)$ lifts to an integral class $c_1 \in H^2(N)$.
In summary, we have
\begin{equation} \label{eq:c1}
v_2(N) = w_2(N) = \rho_2(c_1) \in H^2(N; \Z/2).
\end{equation}
%
By definition of the Wu class $v_2(N)$ we have
\begin{equation} \label{eq:Cartan}
xy v_2(N) = Sq^2(xy) = x^2y + xy^2,
\end{equation}
where we have used the Cartan formula for $Sq^2(x y)$ and the
fact that $Sq^1(\rho_2(y)) = 0$.
By Equation \eqref{eq:c1} we can replace $v_2(N)$ by $c_1$ in
\eqref{eq:Cartan} and rearranging we obtain
\begin{equation} \label{eq:Wu}
x^2y = xyc_1 + xy^2.
\end{equation}
By Example \ref{ex:lambda-x}, there is an element $d_x \in TH^2(N)$ such that
$xy^2 = xyd_x$ and so
$$x^2y = xyc_1 + xyd_x = xye_x,$$
where $e_x : = c_1 + d_x$.
Hence we have found $e_x$ as in \eqref{eq:key},
finishing the proof of Theorem \ref{thm:main}.

%

\section{Teichner's examples} \label{s:TE}
In this section we recall a construction due to Teichner \cite{Teichner}, which produces closed smooth $6$-manifolds $N$ with classes $x\in H^2(N;\Z/2)$ such that
$\beta^{\Z/2}(x^2)\neq 0$.
The manifolds $N$ are constructed as total spaces of sphere bundles of rank $3$ vector bundles $E$ over closed $4$-manifolds.
In the following, $\Z^{w_1(E)}$ denotes integral coefficients
twisted by the first Stiefel-Whitney class of the bundle $E$.

\newpage

\begin{lem}[{\cite[Lemma 1]{Teichner}}] \label{lem:T1}
Let $E$ be a $3$-dimensional bundle over a path-connected space $X$, with sphere bundle $N=SE$.
\begin{enumerate}
\item[(i)] There exists a class $x\in H^2(N;\Z/2)$ which restricts to the generator in the cohomology $H^2(S^2;\Z/2)$ of the fibre if and only if $w_3(E)=0$.
\item[(ii)] Assume that $w_2(E)$ is not the reduction of a class in $H^2(X;\Z^{w_1(E)})$. Then any class $x$ as in (i) has $0\neq \beta^{\Z/2}(x^2)\in H^5(N;\Z)$.
\end{enumerate}
\end{lem}

The next lemma guarantees the existence of such bundles with base $X=M$ a closed connected $4$-manifold.

\begin{lem}[{\cite[Lemma 2]{Teichner}}] \label{lem:T2}
Let $M$ be a closed connected $4$-manifold with fundamental group $\Z/4$. Then there exists a $3$-dimensional bundle $E$ over $M$ with $w_3(E)=0$, $w_1(E)= w_1(M)$ and $w_2(E)$ not the reduction of a class in $H^2(M;\Z^{w_1(E)})$. \qed
\end{lem}

\begin{defn} \label{defn:TP}
The total space $N$ of the sphere bundle of a bundle $E$ satisfying the conditions of Lemma \ref{lem:T2} is a closed connected $6$-manifold, which by Lemma \ref{lem:T1} supports a class $x\in H^2(N;\Z/2)$ satisfying $\bz(x^2)\neq 0$. We will call such a total space $N$ a {\em Teichner manifold}
and the pair $(N, x)$ a {\em Teichner pair}.
\end{defn}


\subsection{\spc $6$-manifolds $N$ with $\bz(x^2) \neq 0 \in H^5(N)$}
In this subsection we show that a Teichner manifold over an orientable base is \spc.

\begin{lem} \label{lem:spinc}
%
Let $N$ be a Teichner manifold over a closed connected $4$-manifold $M$. Then:
\begin{enumerate}
\item[(i)] $N$ is orientable;
\item[(ii)] if $M$ is orientable, then $N$ is \spc.
\end{enumerate}
\end{lem}

\begin{proof}
Let $\pi:N\to M$ be the bundle projection. Since the normal bundle of the sphere bundle in the total space of $E$ is trivial, there are bundle isomorphisms
$$
TN\oplus \R \cong TE|_N \cong \pi^*(TM)\oplus\pi^*(E).
$$
Now part (i) follows from the equation
\[
w_1(N) = \pi^*w_1(M) + \pi^*w_1(E) = 0.
\]
For (ii), assume $w_1(M)=0$ so that
\[
w_2(N) = \pi^*w_2(M) + \pi^*w_1(M)\pi^*w_1(E) + \pi^*w_2(E) = \pi^*w_2(M)+\pi^*w_2(E).
\]
Then
\[
\bz(w_2(N)) = \pi^*(\bz(w_2(M)) + \pi^*(\bz(w_2(E))).
\]
The first term vanishes since any orientable $4$-manifold is \spc; see \cite{Morgan} for example. The second term vanishes since $\bz(w_2(E))\in H^3(M)$ is the Euler class of the orientable bundle $E$.
\end{proof}

The following proposition proves Theorem \ref{thm:T1}(i).

\begin{prop} \label{prop:tbeta}
Let $(N, x)$ be a Teichner pair over a closed, connected, orientable $4$-manifold.
Then $N$ is \spc and $\bz(x^2) \neq 0$, but $\bz(x^2) \in \bz(x)H^2(N)$.

Furthermore, the element $\alpha = \bz(x)\in TH^3(N)$ has $\per(\alpha)=2$
and $\ind(\alpha)=4$.
\end{prop}

\begin{proof}
The first statement is a consequence of Lemma \ref{lem:spinc}, Lemma \ref{lem:T1}
and Theorem \ref{thm:main}.

To prove the second statement,
we recall that by \cite[Theorem A]{AW2},
$$\ind(\alpha) = \ord(\wt Q(x)) \per(\alpha),$$
where $\wt Q(x)=[\beta^{\mathbb{Z}/4}(P_2(x))]\in H^5(N)/\alpha H^2(N)$.
Note that by Theorem \ref{thm:main} and \eqref{eq:beta},
\[
2\wt Q(x) = 2[\beta^{\mathbb{Z}/4}(P_2(x))] = [2\beta^{\mathbb{Z}/4}(P_2(x))] = [\bz(x^2)]=0,
\]
since $N$ is \spc.
However $\wt Q(x)\neq 0$, since any element of $\alpha H^2(N)$ is $2$-torsion, while
\[
2\beta^{\Z/4}(P_2(x)) = \bz(x^2)\neq 0.
\]
Hence $\ord(\wt Q(x))=2$ and we're done.
\end{proof}
%

\subsection{$6$-manifolds violating the TPIC}
In this subsection we give examples of Teichner pairs $(N, x)$
over a non-orientable base which violate the topological period-index conjecture;
i.e.~$\bz(x^2) \notin \bz(x)H^2(N)$.
%
%
We first prove
an extension
of \cite[Lemma 2]{Teichner}.

\begin{lem}\label{lem:TeichnerExtension}
Let $M$ be a closed connected $4$-manifold with an element $a\in H_1(M)$ of order $4$. Then there exists a $3$-dimensional bundle $E$ over $M$ with $w_1(E)= w_1(M)$, $w_2(E)$ not coming from $H^2(M;\Z^{w_1(E)})$ and $w_3(E)=0$.
\end{lem}
\begin{proof}
We use multiplicative notation for elements of $H_1(M)=\pi_1(M)_{\rm ab}$. The Poincar\'e dual of $a^2$ in $H^3(M;\Z^{w_1(M)})$ has order $2$, hence is the image of an element $z\in H^2(M;\Z/2)$ under the twisted Bockstein. As in Teichner's proof of \cite[Lemma 2]{Teichner},
there are no obstructions to constructing a $3$-bundle $E$ with $(w_1(E),w_2(E))=(w_1(M),z)$.

It remains to show that $w_3(E)=0$. This follows from Theorem 2.3 of \cite{Greenblatt}, which states that for any space $X$ and twisting $w\in H^1(X;\Z/2)$, the composition of the twisted Bockstein $\beta^w:H^i(X;\Z/2)\to H^{i+1}(X;\Z^w)$ with reduction mod $2$ is given by
\[
\rho_2 \circ \beta^w(z)=Sq^1(z) + zw.
\]
Hence we have
\begin{align*}
w_3(E) & = Sq^1(w_2(E)) + w_2(E) w_1(E) \\
      & = \rho_2 \circ \beta^{w_1(M)}(w_2(E))\\
      & = 0,
\end{align*}
since $\beta^{w_1(M)}(w_2(E)) = \beta^{w_1(M)}(z)$ is even.
\end{proof}

In order to find an example with $\bz(x^2) \notin \bz(x)H^2(N)$
it turns out to be sufficient that there is an element $a\in H_1(M)$ of order $4$ such that
$0\neq\tau_!(a^2)\in H_1(\hat{M})$, where $\tau_!:H_1(M)\to H_1(\hat{M})$ is the transfer associated to the orientation double cover $\tau:\hat{M}\to M$.

To this end, we shall use a closed connected $4$-manifold $M$ with
\[
\pi_1(M)=C_8\rtimes C_2 = \langle a, b \mid a^b=a^5,a^8,b^2\rangle
\]
and with $w_1(M):\pi_1(M)\to C_2$ the projection onto the base of the semi-direct product. Note that
\[
H_1(M) = \langle a,b \mid a=a^5,a^8,b^2,[a,b]\rangle \cong C_4\times C_2
\]
has an element $a$ of order $4$.
It is well-know, see e.g.~\cite[Propostion 11.75]{Ranicki}, that every
homomorphism $w \colon \pi \to \Z/2$ from a finitely
presented group $\pi$ arises as $(\pi_1(X), w_1(X))$ for a $4$-manifold $X$,
and so a $4$-manifold $M$ as above exists.

\begin{lem} \label{lem:tau}
The transfer homomorphism
$\tau_!:H_1(M)\to H_1(\hat{M})$ does not map the element $a^2\in H_1(M)$ to $0$.
\end{lem}
\begin{proof}
Let $G=\pi_1(M)$ and let $H=\ker(w_1(M)) = C_8$, so that $[G:H]=2$. The definition of the transfer in terms of coset representatives gives
\[
\tau_!: G_{\rm ab} \to H_{\rm ab},\qquad g[G,G]\mapsto g^2[H,H].
\]
Therefore $\tau_!(a^2)=a^4\neq 0$ as claimed.
\end{proof}

Before continuing, we record the following lemma which will be useful in the proof
of Proposition \ref{prop:T} below.

\begin{lem}[{\cite[VII.8.10]{Dold}}] \label{lem:connecting}
Let $i:A\to X$ denote the inclusion of a CW-pair $(X,A)$, and let $\delta: H^*(A)\to H^{*+1}(X,A)$ be
the connecting homomorphism in the long exact cohomology sequence (with any coefficients). Then for all $x\in H^*(A)$ and $y\in H^*(X)$ we have
\[
\pushQED{\qed}
\delta(x i^*(y)) = \delta(x) y. \qedhere
\popQED
\]
\end{lem}

The following proposition proves Theorem \ref{thm:T1}(ii).

\begin{prop} \label{prop:T}
Let $(N, x)$ be a Teichner pair over a non-orientable $4$-manifold
$M$ with $w_1(M) \colon \pi_1(M) \to \Z/2$ as above.
Then $\bz(x^2) \notin \bz(x)H^2(N)$.

Furthermore, the element $\alpha = \bz(x)\in TH^3(N)$ has $\per(\alpha)=2$
and $\ind(\alpha)=8$.
\end{prop}

\begin{proof}
We first prove that $\bz(x^2) \notin \bz(x)H^2(N)$.
Suppose towards a contradiction that $\beta^{\Z/2}(x^2) = \beta^{\Z/2}(x) Y$ for some $Y\in H^2(N)$. Let $i:N\hookrightarrow DE$ be the inclusion of the unit sphere bundle in the unit disc bundle of $E$. From the long exact sequence of the pair $(DE,N)$, the twisted Thom isomorphism $H^3(DE,N)\cong H^0(M;\Z^w)$ and the fact that $M$ is non-orientable, we see that $i^*:H^2(DE)\to H^2(N)$ is surjective.  Hence $Y=i^*(y)$ for some $y\in H^2(DE)\cong H^2(M)$.

Let $t_E^w\in H^3(DE,N;\Z^w)$ be the twisted Thom class of $E$, and
$t_E\in H^3(DE,N;\Z/2)$ its mod $2$ reduction. From the fact that $x$ restricts to a generator in each fibre, it follows that $t_E=\delta(x)$, where $\delta: H^*(N;\Z/2)\to H^{*+1}(DE,N;\Z/2)$ is the connecting homomorphism (see the proof of Lemma 1 in \cite{Teichner}).
%
Now we have
\[
\delta(x^2) = \delta(Sq^2(x)) = Sq^2(\delta(x)) = Sq^2(t_E) = w_2(E) t_E
\]
and since Bocksteins commute with connecting homomorphisms
\[
\delta(\bz(x^2)) = \bz(\delta(x^2)) = \bz(w_2(E)t_E).
\]
On the other hand, $\bz(x^2) = \bz(x)i^*(y)$ and so
\begin{align*}
 \delta(\bz(x^2)) &  = \delta(\bz(x)i^*(y)) \\
                 & = \delta(\bz(x i^*(\rho_2(y)))) \\
                 & = \bz(\delta (x i^*(\rho_2(y)))) \\
                 & = \bz(\delta(x)\rho_2(y)) \\
                 & = \bz(t_E \rho_2(y)).
 \end{align*}
 Here we have used Lemma \ref{lem:Bprod} and Lemma \ref{lem:connecting}.

The above shows that $\bz(w_2(E) t_E) = \bz(t_E \rho_2(y))$, or equivalently $t_E(w_2(E)-\rho_2(y))$ is the reduction of an integral class. From the square
\[
\xymatrix{
H^5(DE,N) \ar[r]^-{\rho_2} & H^5(DE,N;\Z/2) \\
H^2(M;\Z^w) \ar[u]^{\cup t_E^w}_\cong \ar[r]^-{\rho_2} & H^2(M;\Z/2), \ar[u]^{\cup t_E}_\cong
}
\]
which commutes since the Thom isomorphisms commute with reduction mod $2$, we see that $w_2(E)-\rho_2(y)$ is the reduction of a twisted integral class, or equivalently
\[
\beta^w(w_2(E))=\beta^w(\rho_2(y)).
\]

Next we lift this equation to the orientation cover, using the commutative square
\[
\xymatrix{
H^2(\hat{M};\Z/2) \ar[r]^(0.575){\beta^{\Z/2}} & H^3(\hat{M}) \\
H^2(M;\Z/2) \ar[u]^{\tau^*} \ar[r]^-{\beta^w} & H^3(M;\Z^w) \ar[u]^{\tau^*}
}
\]
to conclude that
\[
\tau^*\beta^w(w_2(E)) = \tau^*\beta^w(\rho_2(y)) = \beta^{\Z/2}(\tau^*(\rho_2(y))) = \beta^{\Z/2}\rho_2(\tau^*(y)) = 0.
\]
However, Poincar\'e duality gives a commutative square
\[
\xymatrix{
H^3(M;\Z^w) \ar[r]^-{\tau^*} \ar[d]_{\cap [M]_w}^\cong & H^3(\hat{M})\ar[d]_{\cap[\hat{M}]}^\cong \\
H_1(M) \ar[r]^-{\tau_!} & H_1(\hat{M}) .
}
\]
 Since the bundle $E$ was chosen as in Lemma \ref{lem:TeichnerExtension} so that $\beta^w(w_2(E))\cap [M]_w=a^2$, and
 $\tau_!(a^2)\neq 0$ by Lemma \ref{lem:tau},
 we see that $\tau^*\beta^w(w_2(E))\neq 0$, a contradiction.

To prove the second statement, we have
$\per(\alpha) = 2$
and since $\bz(x^2) \notin \bz(x)H^2(N)$,
$$
2 \wt Q(x)=[\beta^{\mathbb{Z}/4}(P_2(x))] =
[2\beta^{\Z/4}(P_2(x))] =
[\bz(x^2)] \neq 0.$$
Hence $\ord(\wt Q(x)) = 4$.  As $\ind(\alpha) = \ord(\wt Q(x)) \per(\alpha)$
by \cite[Theorem A]{AW2},
$\ind(\alpha) = 8$.
 %
 \end{proof}
%


%

%



\begin{thebibliography}{99}
\bibitem{AW1} B. Antieau, B. Williams, {\it The period-index problem for twisted topological
$K$-theory,} Geom.\ Topol.\ {\bf 18} (2014), 1115--1148.

\bibitem{AW2} B. Antieau, B. Williams, {\it The topological period-index problem over 6-complexes,} J.\ Topol.\ {\bf 7} (2014), 617--640.


\bibitem{Brown} K. S. Brown, {\it Cohomology of Groups,} Graduate Texts in Mathematics, 87. Springer-Verlag, New York, 1982.


\bibitem{CT} J-L. Colliot-Th\'{e}l\`{e}ne,
{\it Exposant et indice d’alg\`ebres simples centrales non ramifi\'ees,} Enseign.\ Math.\
{\bf 48} (2002), 127--146.

\bibitem{Da-K} J. F. Davis and P. Kirk, 
{\em Lecture notes in algebraic topology}, 
Graduate Studies in Mathematics, 35. American Mathematical Society, Providence, RI, 2001. 

\bibitem{Dold} A. Dold, {\it Lectures on algebraic topology,} $2^{\rm nd}$ edition. Grundlehren der Mathematischen Wissenschaften {\bf 200}, Springer-Verlag, Berlin-New York, 1980.

\bibitem{D-K} P. Donovan, M. Karoubi, {\it Graded Brauer groups and {$K$}–theory with local coefficients,}
Inst.\ Hautes \'{E}tudes Sci.\ Publ.\ Math.\ (1970) 5--25.


\bibitem{G-Sz} M. Grant, A. Sz\H ucs, {\it On realizing homology classes by maps of restricted complexity,} Bull.\ London Math.\ Soc.\ {\bf 45} (2013), 329--340.

\bibitem{G} A. Grothendieck, {\it Le groupe de Brauer. I. Alg\`{e}bres d'Azumaya et interpr\'{e}tations diverses} (French) Dix expos\'{e}s sur la cohomologie des sch\'{e}mas, 46--66,
Adv.\ Stud.\ Pure Math., {\bf 3}, North-Holland, Amsterdam, 1968.

\bibitem{Greenblatt} R. Greenblatt, {\it Homology with local coefficients and characteristic classes,} Homology Homotopy Appl.\ {\bf 8} (2006), 91--103.

\bibitem{Gu} X. Gu, {\it The Topological Period-Index Problem over $8$-Complexes, I,} Journal of Topology, {\bf 12} (2019), 1368--1395.


\bibitem{deJong}
A. J. de Jong, {\em The period–index problem for the Brauer group of an algebraic surface},
Duke Math.\ J.\ {\bf 123} (2004), 71--94.

\bibitem{Massey} W. S. Massey, {\it Obstructions to the existence of almost complex structure},
Bull.\ Amer.\ Math.\ Soc.\ {\bf 67} (1961), 55--564.

\bibitem{Morgan} J. W. Morgan,
{\it The Seiberg-Witten equations and applications to the topology of smooth four-manifolds},
Mathematical Notes 44, Princeton University Press, 1996.

\bibitem{Ranicki} A.~A.~Ranicki, {\em Algebraic and geometric surgery},
Oxford Mathematical Monographs.
Oxford University Press,
Oxford, 2002.

\bibitem{S-T}  H. Seifert, W. Threlfall, {\it Lehrbuch der Topologie}, Teubner Verlag, 1934.


\bibitem{Teichner} P. Teichner, {\it $6$-dimensional manifolds without totally algebraic homology,} Proc.\ Amer.\ Math.\ Soc.\ {\bf 123} (1995), 2909--2914.

\bibitem{Thom} R. Thom, \textit{Quelques propri\'et\'es globales des vari\'et\'es diff\'erentiables,} Comment.\ Math.\ Helv.\ {\bf 28} (1954), 17--86.

\end{thebibliography}
\end{document}